\documentclass[12pt]{amsart}

\usepackage{fullpage}
\usepackage{enumerate}
\usepackage{amssymb}
\allowdisplaybreaks[4]

\newtheorem{theorem}{Theorem}[section]
\newtheorem{lemma}[theorem]{Lemma}
\newtheorem{corollary}[theorem]{Corollary}
\newtheorem{proposition}[theorem]{Proposition}
\theoremstyle{definition}
\newtheorem{example}[theorem]{Example}

\newtheorem{problem}[theorem]{Problem}
\newtheorem{definition}[theorem]{Definition}
\numberwithin{equation}{section}

\newcommand{\Z}{\ensuremath{\mathbb{Z}}}
\newcommand{\C}{\ensuremath{\mathbb{C}}}
\newcommand{\R}{\ensuremath{\mathbb{R}}}
\newcommand{\T}{\ensuremath{\mathbb{T}}}
\newcommand{\N}{\ensuremath{\mathbb{N}}}

\begin{document}
\title{Wavelets on compact abelian groups}

\author{Marcin Bownik}
\address{(M. Bownik) Department of Mathematics, University of Oregon, Eugene OR, USA}
\email{mbownik@uoregon.edu}
\author{Qaiser Jahan}
\address{(Q. Jahan) School of Basic Sciences, Indian Institute of Technology Mandi, 175005, India}
\email{qaiser@iitmandi.ac.in}


\date{\today}
\thanks{The authors are thankful to Jaros\l aw Kwapisz and Ken Ross for useful discussion about this work. The first author was supported in part by the NSF grant DMS-1665056. The second author is supported by Department of Science and Technology, Govt. of India and Indo-U.S. Science and Technology Forum}


\begin{abstract}
Multiresolution analysis (MRA) on a compact abelian group $G$ has been constructed with epimorphism as a dilation operator. We show a characterization of scaling sequences of an MRA on $L^p(G)$, $1\le p<\infty$. With the help of the scaling sequence we construct an orthonormal wavelet basis of $L^2(G)$.
\end{abstract}

\maketitle


\section{Introduction}

In recent years there has been a considerable interest in construction of wavelets on locally compact abelian groups. Dahlke \cite{Dah} was one of the first to introduce the concept of wavelets on locally compact abelian groups as he has constructed MRA and wavelets with the help of self-similar tiles and $B$-splines. Lang \cite{Lang1, Lang2} has constructed wavelets on the Cantor dyadic group.  Wavelets on more general $p$-adic Vilenkin groups were studied by Farkov \cite{Farkov}. 
J. J. Benedetto and R. L Benedetto \cite{BB1, BB2} studied wavelets on local fields and more generally on totally disconnected, nondiscrete locally compact abelian group with compact open subgroup. Wavelets on local fields of zero characteristic, that is a field of $p$-adic numbers, were studied by Skopina and her collaborators \cite{AES, KSS, SS}. Multiresolution analysis and wavelets on local fields of positive characteristic were given by Jiang, Li, and Jin \cite{JLJ} and Behera and Jahan \cite{BJ1, BJ2}. Multiresolution analysis and wavelet bases on abelian zero-dimensional groups were studied by Lukomskii \cite{Luk1, Luk2}, and more recently by Barg and Skriganov \cite{BS} in a general setting of association schemes on measure spaces.

The underlying theme of these works is that we are given an automorphism on a locally compact abelian group $G$ which plays a role of a dilation and a discrete subgroup of $G$ which plays a role of translations. As in the classical setting of wavelets on the real line, or Euclidean space $\R^d$, a wavelet system is generated by translates and dilates of a finite collection of functions in $L^2(G)$ over integer scales. In contrast, when the group $G$ is compact, we can no longer require that a dilation is given by an automorphism, but rather by a surjective endomorphism (epimorphism) of $G$. The reason is that automorphisms of a compact group $G$ do not lead to a sensible definition of an MRA.
This is already seen in the construction of periodic wavelets over a finite dimensional torus $G=\T^d$ by Maksimenko and Skopina \cite{MS}, where the role of dilation is played by an epimorphism of $\T^d$, which is not an automorphism. Consequently, wavelets are indexed only over positive scales since stretching (negative dilates) is not available in the compact case.

In this paper we assume that we are given a compact abelian group $G$ and an epimorphism $A:G \to G$ with a finite kernel such that $\bigcup_{j\in \N_0} \ker A^j$ is dense in $G$. These standing assumptions are necessary to guarantee that an MRA $(V_j)_{j\in \N_0}$ satisfies the density property  $\overline{\bigcup_{j=0}^{\infty}V_j}=L^p(G)$. Inspired by the work of Skopina \cite{Sk} and her collaborators \cite{MS, NPS}, we define the concept of a multiresolution analysis (MRA) in this setting. 
Our first main result is a characterization of scaling sequences of an MRA for $L^p(G)$, $1\le p<\infty$, which generalizes the results of Maksimenko and Skopina \cite{MS} from a finite dimensional torus $\T^d$ to a compact abelian group $G$. The results in \cite{MS} require that an epimorphism $A$ of $\T^d$ is given by an expansive $d\times d$ matrix with integer entries. That is, all eigenvalues $\lambda$ of $A$ satisfy $|\lambda|>1$. Even in the setting of the torus $G=\T^d$, our results are a generalization of \cite{MS} as we impose a weaker assumption on an epimorphism. We show that our standing assumptions in the case of the torus $\T^d$ are equivalent to $A$ having no eigenvalues which are integral algebraic units. That is, for each eigenvalue $\lambda$ of $A$, its reciprocal $1/\lambda$ is not an algebraic integer.
Beyond the setting of the torus we provide several examples of epimorphisms of compact abelian groups satisfying our standing assumptions. These include a compact Cantor group with more general dilations than the backward shift mapping. 

Our second main result shows the existence of minimally supported frequency (MSF) multiresolution analysis for every compact abelian group satisfying our standing assumptions. This is an important result as it shows that our characterization results are not vacuous despite the fact the actual constructions of MRAs need to be customized to a specific group $G$ and an epimorphism $A$. Moreover, once an MRA is given to us, we show that a rather standard procedure yields an orthonormal wavelet basis of $L^2(G)$.

In Section~2, we present the necessary definitions and properties of epimorphisms on compact abelian groups. We also provide several specific examples of compact abelian groups and epimorphisms satisfying our standing assumptions. In addition, we characterize epimorphisms of the torus $\T^d$ with dense kernel of iterates. In Section~3, we define the concept of an MRA $(V_j)_{j\in \N_0}$ on a compact abelian group and we prove the characterization of scaling sequences which is preceded by many results including the construction of a basis in each space $V_j$. In the last section we construct wavelet  bases for $L^2(G)$. We also prove the existence of MSF MRA under our standing assumptions on an epimorphism $A$. We conclude the paper by constructing an orthonormal MSF wavelet basis of $L^2(G)$.

\section{Preliminaries}

In this section we give some basic definitions and set our notations which we will use throughout the article.        
Let $G$ be a second countable locally compact abelian group. Let $\widehat{G}$ be its dual group, i.e., 
\[
\widehat{G} = \{\chi:G\rightarrow \mathbb{C}: \chi~{\rm is~a~continuous~character~of~} G\}
\]
with the additive group operation $(\chi_1+\chi_2)(x)=\chi_1(x)\chi_2(x)$. For convenience we denote identity element of this group as ${\bf 0}$. The following result can be found in \cite{F, R, Rudin}.

\begin{theorem}
If $G$ is compact, then $\widehat{G}$ is discrete. If $G$ is discrete, then $\widehat{G}$ is compact.
\end{theorem}

\begin{definition}
Let $H\subset G$ be a subgroup of $G$. We define the subgroup $H^{\perp}$, called the annihilator of $H$, as the collection of all characters which are trivial on the subgroup $H$,
\[
H^{\perp}=\{\chi\in\widehat{G}:\chi(h)=1~{\rm for~all}~ h\in H\}.
\]
\end{definition}

\begin{definition}
For all $f\in L^1(G)$, the function $\hat{f}$ defined on $\widehat{G}$ by
\[
\hat{f}(\chi)=\int_Gf(x)\overline{\chi(x)}dx	
\]
is called the Fourier transform of $f$. Here, $dx$ denotes a left invariant Haar measure on $G$, which is also right invariant since $G$ is abelian.
\end{definition}
 
We denote $\N=\{1,2,\ldots\}$ and $\N_0=\{0,1,2,\ldots\}$. Let ${\bf Epi}(G)$ be the semigroup of continuous group homomorphism of $G$ onto $G$. Then, we have the following elementary fact.

\begin{proposition}\label{prop:ker}
Let $G$ be a locally compact abelian group and $A\in {\bf Epi}(G)$. Then, the set $\bigcup\limits_{j\in\N_0}\ker A^j$ is dense in $G$ if and only if $\bigcap\limits_{j\in \N_0}\big(\ker A^j\big)^{\perp}=\{\bf 0\}$.
\end{proposition}

\proof Suppose that $\bigcup\limits_{j\in\N_0}\ker A^j$ is dense in $G$. Take $\chi\in\bigcap\limits_{j\in \N_0}\big(\ker A^j\big)^{\perp}$, i.e., $\chi(x)=1$ for all $x \in\ker A^j$ and for all $j\in\N_0$. By continuity, we have $\chi(x)=1$ for all $x\in G$, which implies $\chi={\bf 0}$.

Conversely, suppose $H=\overline{\bigcup\limits_{j\in\N_0}\ker A^j}$ is a proper closed subgroup of $G$. Then, $G/H$ is nontrivial which implies $\widehat{G/H}$ is also nontrivial. By \cite[Theorem 2.1.2]{Rudin}, $\widehat{G/H}=H^{\perp}$ and hence $H^{\perp}$ is also nontrivial. Take ${\bf 0} \ne \chi\in H^{\perp}$. Then, $\chi(x)=1$ for all $ x\in \ker A^j$, $j\geq 0$. This implies $\chi\in \big(\ker A^j\big)^{\perp}$ for all $j\geq 0$, which gives $\chi\in\bigcap\limits_{j\in\N_0}\big(\ker A^j\big)^{\perp}$. Therefore, $\bigcap\limits_{j\in\N_0}\big(\ker A^j\big)^{\perp}\neq \{\bf 0\}$. 
\qed

As in \cite{BR}, let ${\bf Epick}(G)$ be the collection of all $A\in {\bf Epi}(G)$ having compact kernel. Given $G$, ${\bf Epick}(G)$ is a semigroup under composition. Moreover, by \cite[Theorem~6.2]{BR} there is a semigroup homomorphism $\Delta: {\bf Epick}(G)\longrightarrow (0, \infty)$ such that
\begin{equation}\label{homo}
\int_G(f\circ A)(x)dx=\Delta(A)\int_Gf(x)dx
\end{equation}
for all integrable functions $f$ on $G$ with respect to the Haar measure $dx$. To obtain $\Delta(A)$, observe that $f \mapsto \int_G(f\circ A)(x)\, dx$ defines a positive translation-invariant linear functional on the space $C_{c}(G)$ of continuous functions on $G$ with compact support and use the uniqueness of Haar measure up to a normalization \cite[Theorem (15.5)]{R}.  

\begin{definition}\label{dep}
Let $G$ be a locally compact abelian group and $A\in {\bf Epi}(G)$ has a finite kernel. Define the periodization operator $P$ acting on functions $f$ on $G$ by
\[
Pf(x) = \sum\limits_{a\in \ker A}f(y+a) \qquad\text{where }y\in A^{-1}x, x\in G.
\]
\end{definition}

\begin{proposition}\label{lem:integral}
For all integrable functions $f$ on $G$, the periodization operator $P$ satisfies 
\[
\int_GPf(x)dx= |\ker A| (\Delta(A))^{-1}\int_Gf(x)dx.
\]
\end{proposition}

\proof By (\ref{homo}), we have
\begin{equation}\label{eqn:ker1}
\int_GPf(Ax)dx= \Delta(A)\int_GPf(x)dx.
\end{equation}
Using the translation invariance of Haar measure, we have
\begin{equation}\label{eqn:ker2}
\int_GPf(Ax)dx = \sum\limits_{a\in \ker A}\int_Gf(x+a)dx= |\ker A| \int_Gf(x)dx.
\end{equation}
The result follows from equations~(\ref{eqn:ker1}) and (\ref{eqn:ker2}).
\qed

In this article we mainly concentrate on compact abelian groups.  If $G$ is compact, then by taking $f \equiv 1$, we deduce that  for any epimorphism $A$ we have $\Delta(A)=1$ in Proposition \ref{lem:integral}. In fact, the {\bf standing assumptions} in the paper are that:
\begin{itemize}
\item $G$ is a compact abelian group, 
\item $A\in {\bf Epi}(G)$ has a finite kernel, 
\item $\bigcup\limits_{j\in\N_0}\ker A^j$ is dense in $G$.
\end{itemize}

First, we will consider the classical case when $G$ is a finite dimensional torus $\T^d= \R^d/\Z^d$. 
Let $A$ be a $d\times d$ matrix with integer entries. Then, $A$ induces an endomorphism $T=T_A$ of $\T^d = \R^d/\Z^d$, and every endomorphism of $\T^d$ is induced in this way. Moreover, $T_A$ is an epimorphism of $\T^d$ if and only if $A$ is an invertible matrix, see \cite[Theorem 0.15]{Wal}.
The following result, which was communicated to the authors by J. Kwapisz, classifies all epimorphisms on $\T^d$ satisfying our standing assumptions. 

\begin{theorem}\label{tdc1}
Let $A$ be a $d\times d$ invertible matrix with integer entries. Suppose that $T_A$ is a surjective endomorphism (epimorphism) on $\T^d$. Then the following are equivalent:
\begin{enumerate}[(i)]
\item
$\overline{\{x \in \T^d:  (T_A)^n x = 0 \text{ for some }n \geq 0\}} \ne \T^d$,
\item
$A$ has an eigenvalue $\lambda \in \C$ which is an integral algebraic unit, i.e., both $\lambda$ and $1/\lambda$ are algebraic integers.
\end{enumerate}
\end{theorem}

We were unable to find Theorem \ref{tdc1} in the literature and hence we present its proof. First, we need to show a basic lemma.

\begin{lemma}\label{tdc2}
Let $A$ be a $d\times d$ invertible matrix with integer entries. Suppose that $T_A$ is a surjective endomorphism (epimorphism) on $\T^d$ and $K\subset \T^d$ is a set. Then 
$$\overline{T^{-1}_A(K)}=T^{-1}_A(\overline{K})$$
\end{lemma}

\begin{proof}
One side of the inclusion is obvious, i.e., $\overline{T^{-1}_A(K)}\subset T^{-1}_A(\overline{K})$ since $T^{-1}_A(K)\subset T^{-1}_A(\overline{K})$. We claim that 
$$T^{-1}_A(\overline{K})\subset \overline{T^{-1}_A(K)}.$$
To prove this, let $x\in T^{-1}_A(\overline{K})$. Then we have $y\in \overline{K}$ such that $y=T_Ax$. There exists a sequence $(y_n)$ in $K$ which converges to $y$. Since $T_A$ is a local homeomorphism, there exists a neighborhood $W$ of $x$ and a neighborhood $U$ of $y$ such that $T_A|_W: W \to U$ is a homeomorphism. Hence, $(T_A|_W)^{-1}y_n$ converges to $(T_A|_W)^{-1}y$. Since $(T_A|_W)^{-1}y=x$, we have $x\in \overline{(T_A|_W)^{-1}(K)}\subset \overline{T_A^{-1}(K)}$.
\end{proof}
\vspace{.3cm}

\begin{proof}[Proof of Theorem~\ref{tdc1}.]
Since $A$ is invertible, it induces a surjective endomorphism (epimorphism) $T_A$ on $\mathbb T^d$. Let $q=|\det A|$. Then, $T_A$ is $q$-to-$1$ mapping. That is, for every $x\in \T^d$, $(T_A)^{-1}(x)$ consists of $q$ points. Also $T_A$ is a  local homeomorphism. If $q=1$, then the result is trivial since we necessarily have $\ker T_A=\{0\}$ and all eigenvalues of $A$ are integral algebraic units. Hence, we can assume that $q\ge 2$.

Let 
$$H:=\overline{\{x \in \T^d:  (T_A)^n x = 0 \text{ for some }n \geq 0\}}.$$
Then, $H$ is a closed subgroup of $\T^d$ and $(T_A)^{-1}(H)=H$, by Lemma~\ref{tdc2}.
Let $H_0$ be the connected component of $H$ containing $0$.
Thus, $H_0$ is a closed connected subgroup of $\T^d$, hence a subtorus. Moreover, $G:=H/H_0$ is a discrete compact group on which $T_A$ induces a surjective endomorphism, hence an automorphism. We also have $T_A(H_0)=H_0$. 

We claim that 
\begin{equation}\label{tdc20}
(T_A)^{-1}(H_0)=H_0.
\end{equation} Suppose that $H_0 \ne (T_A)^{-1}(H_0)$. Since $H_0 \subset (T_A)^{-1}(H_0) \subset H$, there exists $h\in H \setminus H_0$ such that $h+H_0 \subset (T_A)^{-1}(H_0)$. Thus, $T_A(h+H_0) \subset H_0$, which contradicts the fact that $T_A$ is an automorphism on $G$.

A subtorus $H_0 \subset \T^d$ lifts to a rational $A$ invariant linear subspace $K_0 \subset \R^d$, i.e., $K_0$ a linear span of rational vectors and $A(K_0) = K_0$. The formula \eqref{tdc20} implies that the endomorphism $T_A$ restricted to the subtorus $H_0$ is $q$-to-$1$ mapping. Consequently, the linear map $A$ restricted to $K_0$ has determinant $\pm q$. 

The matrix $A$ also induces a linear mapping $\tilde A: \R^d/K_0 \to \R^d/K_0$, which corresponds to endomorphism of the torus $\T^d/H_0$. Hence, $\tilde A$ can be identified with an integer matrix, see \cite[Theorem~0.15]{Wal}. The characteristic polynomial of $A$ is the product of characteristic polynomials of $A|_{K_0}$ and $\tilde A$. These polynomials have all integer coefficients. Since the constant coefficients of $A$ and $A|_{K_0}$ are $\pm q$, the characteristic polynomial of $\tilde A$ is an integral monic polynomial with the constant term $\pm 1$. This proves $(i) \implies (ii)$.

To prove the converse implication we assume $(ii)$. Thus, the characteristic polynomial $p \in \Z[x]$ of $A$ is divisible by a monic polynomial $p_0\in \Z[x]$ with constant coefficient $\pm 1$. Hence, $p_1:=p/p_0 \in \Z[x]$ is a monic polynomial with constant coefficient $\pm q$. Consider the invariant subspaces $K_0$ and $K_1$ corresponding to $p_0$ and $p_1$, i.e.,
\[
K_0= \{x\in \R^d: p_0[A]x =0 \}, \qquad
K_1= \{x\in \R^d: p_1[A]x =0 \}.
\]
Then, $K_0$ and $K_1$ are rational subspaces of $\R^d$ which are invariant under $A$. Moreover, the characteristic polynomial of $A$ restricted to $K_i$ is $p_i$, $i=0,1$. The matrix $A$ has a block diagonal form with respect to subspaces $K_0$ and $K_1$. So does any power $A^n$, $n\ge 1$. 
Let $H_i$ be a subtorus of $\T^d$ corresponding to a subspace $K_i$, $i=0,1$. Since $A|_{K_0}$ has determinant $\pm 1$, $T_A|_{H_0}$ is an automorphism of $H_0$. Hence, $\ker T_A \subset H_1$. Likewise, $\ker (T_A)^n \subset H_1$ for any $n\ge 1$. Since $H_1$ is a proper subtorus, this yields $(i)$.
\end{proof}

As a corollary of Theorem \ref{tdc1} we obtain

\begin{corollary}\label{cori}
Let $A$ be a $d\times d$ invertible matrix with integer entries such that no eigenvalues of $A$ are  integral algebraic units. Then the epimorphism $T_A$ satisfies our standing assumption, i.e.,
\begin{equation}\label{dense}
\overline{\{x \in \T^d:  (T_A)^n x = 0 \text{ for some }n \geq 0\}} = \T^d.
\end{equation}
In particular, for  any expansive matrix $A$, i.e., all its eigenvalues $\lambda$ of $A$ satisfy $|\lambda|>1$, the corresponding epimorphism $T_A$ satisfies \eqref{dense}.
\end{corollary}

\proof If $A$ is a $d\times d$ invertible matrix with integer entries such that no eigenvalues of $A$ are  integral algebraic units then by Theorem \ref{tdc1}, \eqref{dense} holds. To prove the second part of the corollary, assume that $A$ is expansive. But suppose that \eqref{dense} fails, i.e., 
\[
\overline{\{x \in \T^d:  (T_A)^n x = 0 \text{ for some }n \geq 0\}} \ne \T^d.
\]
Then by Theorem \ref{tdc1}, $A$ has an eigenvalue $\lambda\in\C$ which is an integral algebraic unit, i.e., both $\lambda$ and $\frac{1}{\lambda}$ are algebraic integers. Hence, the characteristic polynomial of $A$ is divisible by the minimal monic polynomial $p$ of $\lambda$, which has integer coefficients. Since $1/\lambda$ is also an algebraic integer, the constant coefficient of $p$ is $\pm 1$. Hence, the product of eigenvalues of $A$, which correspond to the roots of $p$, is equal to $\pm 1$. This gives a contradiction with the fact that $A$ is expansive.
\qed

The well-known doubling map illustrates the essence of our standing assumptions.

\begin{example}
Let $G=\T =\R/\Z$. Let $m$ be an integer such that $|m| \ge 2$. Define an epimorphism $A:\T \to \T$ as a multiplication map $A(x)=mx \mod 1$, $x\in \T$. Then, $\ker A$ is finite, has cardinality $|m|$, and for any $j\in\N$, 
\[
\ker A^j = \{k/m^j + \Z: k =0, 1,\ldots, |m|^j-1\}.
\]
Hence, the pair $(G,A)$ satisfies the standing assumptions. In particular, when $m=2$, then $A:\T \to \T$ is a well-known doubling map $A(x)=2x \mod 1$.
\end{example}

Next we give more examples of epimorphisms on compact abelian groups satisfying our standing hypothesis.

\begin{example}
For a fixed natural number $N\ge 2$, let $\Z_N=\frac{1}{N}\Z/\Z\simeq\{0, \frac{1}{N}, \frac{2}{N}, \ldots, \frac{N-1}{N}\}$. 
Consider $G=(\Z_N)^{\N}$ equipped with the product topology. By Tychonoff's Theorem $G$ is compact. We define the backward shift mapping $S$ on $G$, i.e., $S(x_1,x_2,\ldots)=(x_2,x_3,\ldots)$. It is straightforward to verify that $S$ satisfies the standing assumptions. In fact, we have a more general example below.
\end{example}

\begin{example}
Consider again $G=(\Z_N)^{\N}$, for fixed natural number $N\ge 2$. Let $A$ be an upper triangular matrix such that main diagonal elements are zero, the upper diagonal elements are 1, and $A$ is the band matrix with upper bandwidth $k\in\N$.  More precisely, 
\begin{equation}\label{bmatrix}
A=\begin{bmatrix}
0 & 1 & a_{1,3} & a_{1,4}  & \dots&a_{1,k+1}&0&0&0&\dots  \\
0 & 0 & 1 & a_{2,4} & a_{2,5} & \dots  &a_{2,k+2}&0&0&\dots  \\
0 & 0 & 0 & 1& a_{3,5} &a_{3,6} & \dots&a_{3,k+3} &0&\dots  \\
\vdots & \vdots &\vdots & \ddots & \ddots&\ddots&\cdots &\ddots&\ddots&\cdots
\end{bmatrix}.
\end{equation}
With the help of the above matrix $A$, we define a homomorphism $T_A$ on $G$ by
\[
T_A(Y)=AY
\]
where $Y=(y_1, y_2, y_3, \ldots)\in G$ and $AY=\big(\sum\limits_{j=1}^{\infty}a_{1j}y_j, \sum\limits_{j=1}^{\infty}a_{2j}y_j, \ldots\big)$. 
The following lemma shows that $T_A$ satisfies our standing assumptions.
\end{example}

\begin{lemma}\label{bandmatrix}
	Let $G=(\Z_N)^{\N}$. Suppose that $A$ is an $\N\times\N$ matrix with integer entries such that each row has finitely many non zero entries and $T_A:G\rightarrow G$ is defined by 
	$$T_A(Y)=AY \qquad \text{for } Y \in G.$$ Then
	\begin{enumerate}
		\item[(i)]  $T_A$ is a well defined continuous homomorphism $G \to G$.
		\item[(ii)] If $A$ is of the form (\ref{bmatrix}), then $T_A$ is an epimorphism.
		\item[(iii)] If $A$ is of the form (\ref{bmatrix}), then $\ker T_A$ is finite and its cardinality is bounded by
		\begin{equation}\label{bound}
		|\ker T_A|\le N^{k}.
		\end{equation}
		\end{enumerate}
\end{lemma}

\proof 
Since each row of $A$ has finitely many non-zero entries, $AY$ is well defined for any $Y \in G$, and $T_A$ is a homomorphism. The group $G$ is metrizable with metric given by
\[
d(X,Y)= \sum_{i=1}^{\infty} \frac{|x_i-y_i|}{2^i}, \qquad X=(x_1,x_2,\ldots), Y=(y_1,y_2,\ldots) \in G.
\]
For any $n\in\N$ we can find $m\in \N$ such that $a_{i,j}=0$ for all $1\le i \le n$ and $j>m$. Hence, if $X=(x_1,x_2,\ldots) \in G$ satisfies $x_i=0$ for $1\le i \le m$, then $d(AX,0) \le \sum_{i=n+1}^\infty 2^{-i} =2^{-n}$. Hence, $T_A$ is continuous at $0\in G$ and thus everywhere.

To prove $(ii)$, we define the projection $p_n:G\rightarrow G$ by $$p_n(x_1, x_2, \ldots)=(x_1, x_2, \ldots, x_n, 0, 0, \ldots).$$ We have following two claims:
\begin{itemize}
	\item[] Claim $(a)$: $p_n\circ T_A(G)=p_n(G)$
	\item[] Claim $(b)$: $T_A(G)=G$
	\end{itemize}
To prove Claim $(a)$, take any $Y=(y_1, y_2, \ldots)\in G$. 
By \eqref{bmatrix} for any $X=(x_1, x_2, \ldots)\in G$ we have $$p_n\circ T_A(x_1, x_2, \ldots)=\bigg(x_2+\sum\limits_{j=3}^{k+1}a_{1,j}x_j, \ldots,  x_{n+1}+ \sum\limits_{j=n+2}^{k+n+1}a_{n,j}x_j, 0, \ldots\bigg).$$
We can find $X\in G$ satisfying $p_n \circ T_A(X)=p_n(Y)$ by back substitution. Indeed, let $x_{n+1}=y_n$ and $x_i=0$ for $i>n+1$. Having defined $x_i$ for $i>m$, we let 
\[
x_m=y_{m-1}- \sum\limits_{j=m+1}^{k+m}a_{m-1,j}x_j.
\]

Proof of Claim $(b)$. For fixed $Y\in G$, we find a sequence $(X_n)_{n=1}^{\infty}$ in $G$ such that $$p_n\circ T_A(X_n)=p_n(Y).$$ By the compactness there exists a subsequence $(X_{n_k})$ which converges to $X$ such that $$p_{n_k}\circ T_A(X_{n_k})=p_{n_k}(Y).$$ By continuity of $T_A$, $p_{n_k}\circ T_A(X_{n_k})$ converges to $T_A(X)$ and $p_{n_k}(Y)$ converges to $Y$ as $k\rightarrow\infty$. Hence, we have $$T_A(X)=Y.$$
Proof of $(iii)$. We claim that there are exactly $N^k$ solutions of the equation
\begin{equation}\label{pia}
p_n \circ A (X) =0 \qquad \text{for }X \in p_{n+k}(G).
\end{equation}
Indeed, if we assign values of $x_{n+2}, \ldots, x_{n+k}$, then the value of $x_{n+1}$ is uniquely determined by the $n$'th row of $A$. By back substitution, the values of $x_2,\ldots, x_{n}$ are also uniquely determined. Finally, $x_1$ can take any value in $\Z_N$. Since we can assign $k$ values in $\Z_N$, the number of solutions of \eqref{pia} is $N^k$. 
Since $A$ is a band matrix with bandwidth $k$, if $X\in \ker T_A$, then $p_{n+k}(X)$ is a solution of \eqref{pia}. This implies \eqref{bound}.
\qed

\begin{example}
Consider $G=\T^d\times (\Z_N)^{\N}$, for fixed $N\ge 2$. Let $B$ be a $d\times d$ integer invertible matrix, which induces an epimorphism $T_B$ on $\T^d$. Assume $B$ has no eigenvalues which are algebraic integral units. Let $i: (\Z_N)^\N \to \T^d$ be a homomorphism with a finite image. Let $S: (\Z_N)^\N \to (\Z_N)^\N$ be the backward shift. Define a homomorphism $A$ on $G$ by
\[
A( X, Y) = (T_B(X)+i(Y), S(Y)) \qquad\text{where } X\in\T^d, Y \in(\Z_N)^{\N}.
\]
We claim that $A$ satisfies our standing assumptions. It is easy to show that $A$ is an epimorphism from the fact that $T_B$ and $S$ are both epimorphisms. Moreover, $\ker A$ is finite and its cardinality
\[ |\ker A| =  N |\ker T_B| = N |\det B|.
\]
We only need to prove that $\bigcup\limits_{j\in\N_0}\ker A^j$ is dense in $G$.
A simple calculation yields 
\begin{equation}\label{ex5}
\begin{aligned}
\ker A^j=\bigg\{(X,Y) \in G : y_{j+1}= y_{j+2} & =\dots=0 
\\
&\text{ and } T_B^j(X)=-\sum\limits_{k=0}^{j-1}T_B^k(i(S^{j-k-1}(Y))) \bigg\}.
\end{aligned}
\end{equation}
Take any $(X_0,Y_0)\in G$ such that $Y_0$ has finitely many non-zero coordinates. Hence, $S^j(Y_0)=0$ for sufficiently large $j>j_0$. It suffices to find a sequence $(X_j)_{j\in\N}$ in $\T^d$ such that 
\begin{equation}\label{ex7}
(X_j,Y_0) \in \ker A^j\quad\text{for } j>j_0\qquad\text{and}\qquad \lim_{j\to\infty} X_j = X_0.
\end{equation}
By Theorem~\ref{tdc1}, $\bigcup\limits_{j\in\N_0}\ker (T_B)^j$ is dense in $\T^d$. Therefore, for  every sequence  $(X'_j)_{j=1}^{\infty}$ in $\T^d$, there exists a sequence $(X_j)_{j=1}^{\infty}$ in $\T^d$ such that $X_j\in T_B^{-j}(X'_j)$ and $X_j$ converges to $X_0$ as $j\to \infty$. Taking $X'_j=-\sum\limits_{k=0}^{j-1}T_B^k(i(S^{j-k-1}(Y_0)))$,   this observation and \eqref{ex5} yields \eqref{ex7}.
\end{example}

Despite our efforts, the following problem remains open.

\begin{problem}
	Let $G=\T^{\N}$ be the infinite dimensional torus. Does there exist an epimorphism $A$ on $\T^{\N}$ such that the standing hypotheses on $A$ hold? That is, $\ker A$ is finite and $\bigcup\limits_{j\in\N_0}\ker A^j$ is dense in $\T^{\N}$. 
\end{problem}

\section{MRA and scaling sequences} 

In this section we give the definition of a multiresolution analysis (MRA) in the setting of a compact abelian group $G$ and an epimorphism $A$ satisfying the standing assumptions. Then we give the characterization of scaling functions. Our definition of an MRA is motivated by the definition of a periodic multiresolution analysis due to 
Skopina \cite{Sk} and Maksimenko and Skopina \cite{MS} in higher dimensions; see also \cite{NPS}. However, our definition differs slightly from \cite[Definition 9.1.1]{NPS} since it explicitly mentions a scaling function.

\begin{definition}\label{MRA}
We define the shift operator $T_y$, $y\in G$, acting on functions $f$ on $G$ by  
\[
T_{y}f(x)=f(x-y).
\]
A multiresolution analysis (MRA) of $L^p(G)$ for $1\le p<\infty$ is a sequence $(V_j)_{j\in\N_0}$ of closed subspaces of $L^p(G)$ satisfying the following properties:
\begin{enumerate}
\item[MR1.] $V_j\subset V_{j+1}$ for all $j\in\N_0$,
\item[MR2.] $\overline{\bigcup_{j=0}^{\infty}V_j}=L^p(G)$,
\item[MR3.] $f\in V_j$ if and only if $T_{\gamma}f\in V_j$, for $\gamma\in \ker A^{j}$ and $j\in\N_0$,
\item[MR4.] there exists a function $\varphi_j\in V_j$ such that $(T_a \varphi_j)_{a \in \ker A^j}$ forms a basis of $V_j$, $j\in\N_0$,
\item[MR5.] a) $f\in V_j \Rightarrow f(A(\cdot))\in V_{j+1}$;\\
                   b) $f\in V_{j+1} \Rightarrow Pf\in V_j$, where $P$ is as in Definition \ref{dep}.
\end{enumerate}
A sequence of functions $(\varphi_j)_{j\in\N_0}$ as in MR4 is called a \emph{scaling sequence} of an MRA $(V_j)_{j\in\N_0}$.
\end{definition}

Let $\widehat{A}$ be the adjoint homomorphism to $A$, which is defined by $\widehat{A}(\chi)=\chi\circ A$ for $\chi\in\widehat{G}$. Then, $\widehat{A}$ is a topological isomorphism of $\widehat{G}$ onto the annihilator of $\ker A$, see \cite[Proposition 6.5]{BR}.

\begin{definition}\label{digit}
Any set containing only one representative of each coset, $\widehat{G}/(\ker A)^{\perp} = \widehat G/\widehat A(\widehat G)$, is called a \emph{set of digits} of $A$, which is denoted by $D(A)$. 
Let $m=|\ker A|$ be the cardinality of $D(A)$. 
Then, we define recursively the set $D(A^j)$, $j\in\N$, of representatives of distinct cosets of $\widehat{G}/(\ker A^j)^\perp$ by
\begin{equation}\label{digit1}
D(A^{j+1}) = \{ \widehat A^j \pi + r: r \in D(A^j), \pi \in D(A)\}.
\end{equation}
\end{definition}

To prove that $D(A^{j+1})$ is a set of representatives of distinct cosets of $\widehat{G}/(\widehat A)^{j+1}(\widehat G)$, take any
$\pi, \pi' \in D(A)$ and $r,r' \in D(A^j)$ such that 
\[
\widehat A^j \pi + r - (\widehat A^j \pi' + r') \in (\widehat A)^{j+1}(\widehat G).
\]
We can deduce that $r=r'$ and then $\pi =\pi'$. Hence, elements of $D(A^{j+1})$ represent distinct cosets of $\widehat{G}/(\widehat A)^{j+1}(\widehat G)$. Moreover, its cardinality $|D(A^{j+1})| = |D(A)| |D(A^j)| = m^{j+1}$. Therefore, \eqref{digit1} defines representatives of all such cosets.

The main result of this section is a characterization of scaling functions associated to an MRA $(V_j)_{j\in\N_0}$, which is a generalization of a result of Maksimenko and Skopina \cite[Theorem 7]{MS} to compact abelian groups $G$, see also \cite[Theorem 9.1.4]{NPS}.

\begin{theorem}\label{thm:scaling}
Functions $(\varphi_j)_{j\in \N_0} \subset L^p(G)$ form a scaling sequence for an MRA of $L^p(G)$, $1\le p<\infty$, if and only if:
\begin{enumerate}
\item [(1)] $\widehat{\varphi}_0(\chi)=0$ for all $\chi\neq {\bf 0}$, $\chi\in \widehat{G}$.
\item [(2)] For any $j\in\N_0$ and any $\eta\in \hat{G}$, there exists $\chi\in (\ker A^j)^{\perp}+\eta$ such that $\widehat{\varphi}_j(\chi)\neq 0$.
\item[(3)] For any $\chi\in \widehat{G}$, there exists $j\in\N_0$ such that $\widehat{\varphi}_j(\chi)\neq 0$.
\item[(4)] For any $j\in\N$ and any $\eta\in \widehat{G}$, there exists a number $\mu^j_{\eta}$ such that $\widehat{\varphi}_{j-1}(\chi)=\mu^j_{\eta}\widehat{\varphi}_j(\chi)$ for all $\chi\in (\ker A^j)^{\perp}+\eta$.
\item[(5)] For any $j\in\N_0$ and any $\eta\in \widehat{G}$, there exists a number $\gamma^j_{\eta}\neq 0$ such that $\widehat{\varphi}_{j+1}(\widehat{A}(\chi))=\gamma^j_{\eta}\widehat{\varphi}_{j}(\chi)$ for all $\chi\in (\ker A^j)^{\perp}+\eta$.
\end{enumerate}
\end{theorem}

The proof of Theorem \ref{thm:scaling} follows a similar scheme as in \cite{MS} with necessary changes imposed by the more general setting of this theorem. The following lemmas are useful in proving the main results.

\begin{lemma}\label{lem:V0}
Suppose $V_j\subset L^p(G)$, $1\le p<\infty$, $j\in\N_0$ and axioms MR1, MR2, MR3 and MR5 b) of Definition~\ref{MRA} hold. Then the space $V_0$ consists of constants.
\end{lemma}

\proof The space $V_0$ is one-dimensional by property MR3. Let $f\in V_0$ such that $\|f\|\neq 0$. First we will show that $\hat{f}({\bf 0})\neq 0$. Consider $g=Pf$. By Proposition \ref{lem:integral}
\begin{eqnarray*}
\widehat{g}({\bf 0}) = \widehat{Pf}({\bf 0})
= \int_GPf(x)dx
= |\ker A| \int_Gf(x)dx
=  |\ker A| \widehat{f}({\bf 0}).
\end{eqnarray*}
Let $g_0\in V_j$. Then by MR5 b), $g_1:=Pg_0\in V_{j-1}, \ldots, g_j:=Pg_{j-1}\in V_0$. If we assume $\widehat{f}({\bf 0})=0$, then $\widehat{g_j}({\bf 0})=0$. This implies that any function from any $V_j$ has zero mean, which contradicts axiom MR2 of Definition~\ref{MRA}.

Next suppose that $\widehat{f}(\chi_0)\neq 0$ for some $\chi_0\not = {\bf 0}$. Since $ f\in V_0$, using MR1, we have $f\in V_1$, hence by MR5 b), $g\in V_0$. Since $V_0$ is a one-dimensional space, therefore, for some constant $\lambda$ it follows that $g=\lambda f$ and hence $\widehat{g}(\chi)=\lambda\widehat{f}(\chi)$. From the above calculation $\lambda=|\ker A|$.

We define the $A$-dilation operator on $L^p(G)$ for $1\le p<\infty$ by
\[
D_Af(x)=f(Ax) ~{\rm for}~ {\rm all}~ x\in G.
\]
By \cite[Lemma~6.6]{BR}, we have
\begin{equation*}
\widehat{D_Af}(\chi)=
\left\{
\begin{array}{lll}
\widehat{f}(\widehat{A}^{-1}(\chi)) & {\rm for}~\chi\in \widehat{A}(\widehat{G})=(\ker A)^{\perp},\\
0 & {\rm otherwise}.
\end{array}
\right.
\end{equation*}
Hence, for $\chi\in(\ker A)^{\perp}$, 
\begin{eqnarray*}
\widehat{D_Ag}(\chi) & = & \int_GPf(Ax)\overline{\chi(x)}dx
 =  \sum\limits_{a\in \ker A}\int_Gf(x+a)\overline{\chi(x)}dx\\
&= & \sum\limits_{a\in \ker A}\chi(a)\widehat{f}(\chi)
 =  |\ker A|\widehat{f}(\chi).
\end{eqnarray*}
Therefore, for $\chi\in(\ker A)^{\perp}$, 
\[
\widehat{f}(\chi) = \widehat{f}(\widehat{A}^{-1}(\chi)).
\]
Equivalently, for any $\eta=\widehat{A}^{-1}(\chi)\in\widehat{G}$, we have $\widehat{f}(\widehat{A}\eta)=\widehat{f}(\eta)$. Hence, for any $m\in \N$ we have
\begin{equation}\label{equ}
0\neq \widehat{f}(\chi_0)=\widehat{f}(\widehat{A}\chi_0 )=\cdots =\widehat{f}(\widehat{A}^m\chi_0).
\end{equation}
We claim that $\chi_0, \widehat A\chi_0, \widehat A^2\chi_0, \ldots$ are all distinct. On the contrary, suppose that for some $m\ge 1$ we have
$\chi_0 = \widehat{A}^m\chi_0$. Since $\chi_0(x)=\chi_0(A^mx)$ for all $x\in G$, we necessarily have $\chi_0(x)=1$ for all $x\in \ker A^{km}$, $k\in \N$. By our standing assumptions, Proposition \ref{prop:ker} implies that $\chi_0(x)=1$ for all $x\in G$, which contradicts the assumption that $\chi_0 \ne \bf 0$. 

Combining the above claim with \eqref{equ} leads to the contradiction
with the fact that the Fourier transform maps $L^1(G) \supset L^p(G)$ into $C_0(\widehat{G})$. Consequently, $\widehat{f}(\chi)=0$ for all $\chi \ne \bf 0$, and hence, $f$ is constant.
\qed

\begin{definition}\label{ome}
Define the operators $\omega_{\eta}^j$ on $L^1(G)$, for $j\in\N_0$ and $\eta\in\widehat{G}$, as follows

\begin{eqnarray*}
\omega^0_{\eta}f & := & f,\\
\omega^j_{\eta}f(x) & := & \frac{1}{|\ker A^j|}\sum\limits_{a\in \ker A^j}\overline{\eta(a)}f(x+a).
\end{eqnarray*}
\end{definition}
Note that unlike \cite{NPS}, the operators $\omega_{\eta}^j$ are not defined recursively. 

\begin{lemma}\label{lem:omega}
Let $f\in L^1(G)$, $j\in \N_0$, and $\eta\in\widehat{G}$. 
Then, $\omega^j_\eta$ has a Fourier series representation
\begin{equation}\label{omega}
\omega^j_{\eta}f \sim \sum\limits_{\kappa\in(\ker A^j)^{\perp}}\widehat{f}(\eta+\kappa)(\eta+\kappa).
\end{equation}
That is, for any $\chi \in \widehat{G}$,
\begin{equation}\label{omega3}
\widehat{\omega_{\eta}^jf}(\chi)=\begin{cases}
\widehat{f}(\chi) &\text{if }\chi\in(\ker A^j)^{\perp}+\eta, \\
0 & \text{if }\chi\not\in(\ker A^j)^{\perp}+\eta.
\end{cases}
\end{equation}
In addition, let $V_j\subset L^p(G)$ for $j\in\N_0$ be such that MR3 of Definition~\ref{MRA} holds. If $f\in V_{j_0}$ for fixed $j_0$, then $\omega_{\eta}^jf\in V_{j_0}$ for all $j=0, \ldots, j_0$.
\end{lemma}

\proof 
We start by the following calculation.
\begin{eqnarray*}
\widehat{\omega_{\eta}^jf}(\chi) & = & \frac{1}{|\ker A^j|}\sum\limits_{a\in \ker A^j}\overline{\eta(a)}\int_Gf(x+a)\overline{\chi(x)}dx\\
& = & \frac{1}{|\ker A^j|}\sum\limits_{a\in \ker A^j}\overline{\eta(a)}\chi(a)\int_Gf(x)\overline{\chi(x)}dx.
\end{eqnarray*}
The product of two character is also a character on $\widehat{G}$. Therefore using \cite[Lemma~23.19]{R}, the sum on right hand side is $|\ker A^j|$ if $\chi -\eta\in(\ker A^j)^{\perp}$ and $0$ if $\chi - \eta \not\in(\ker A^j)^{\perp}$. This proves \eqref{omega3}.

Next, suppose that $f\in V_{j_0}$ and $j=0, \ldots, j_0$. Then by MR3 of Definition~\ref{MRA}, $T_af\in V_{j_0}$ for $a\in \ker A^j \subset \ker A^{j_0}$. Therefore, $\omega_{\eta}^jf\in V_{j_0}$.
\qed

\begin{lemma}\label{indep} Let $f\in L^1(G)$ and $j\in \N_0$. Then functions $T_a f$, $a \in \ker A^j$, are linearly independent if and only if $\omega^j_\eta f \ne 0$ for all $\eta \in D(A^j)$.
\end{lemma}

\begin{proof}
Consider $m^j \times m^j$ matrix $(\eta(a))_{\eta \in D(A^j), a\in \ker A^j}$, which represents the discrete Fourier transform of the finite group $\ker A^j \subset G$. Its characters are elements of $\widehat G/(\ker A^j)^\perp$, which we identify with $D(A^j)$. The discrete Fourier transform matrix is a multiple of a unitary matrix, and hence invertible. Therefore, $T_a f$, $a \in \ker A^j$, are linearly independent if and only if $\omega^j_\eta f$, $\eta \in D(A^j)$ are linearly independent. By \eqref{omega3} the supports of $\widehat{\omega^j_\eta f}$, $\eta \in D(A^j)$, are disjoint. Hence, their linear independence is equivalent to $\widehat{\omega^j_\eta f} \ne 0$ for all $\eta \in D(A^j)$.
\end{proof}

\begin{lemma}\label{lem:basis}
Let $(V_j)_{j=0}^{\infty}$ be an MRA of $L^p(G)$, $1\le p<\infty$. Then there exists a family of functions $v^j_\eta$, $j\in \N_0$, $\eta \in \widehat{G}$, satisfying the following properties:
\begin{enumerate}
\item[{\bf V0.}]  $v^j_\eta=v^j_{\eta'}$ if $\eta -\eta' \in (\ker A^j)^\perp$ and $(v^j_{\eta})_{\eta\in D(A^j)}$ is a basis of $V_j$.
\item[{\bf V1.}] $\widehat{v}^j_{\eta}(\chi)=0$ for all $\chi\not\in(\ker A^j)^{\perp}+\eta$.
\item[{\bf V2.}] If \ $\widehat{v}^{j}_{\eta}(\chi_0)\neq 0$ for some $\chi_0\in(\ker A^{j+1})^{\perp}+\eta$, then $\widehat{v}^{j+1}_{\eta}(\chi)=\widehat{v}^{j}_{\eta}(\chi)$ for all $\chi\in(\ker A^{j+1})^{\perp}+\eta$.
\item[{\bf V3.}] $\widehat{v}^{j}_{\eta}(\chi)=\widehat{v}^{j+1}_{\widehat{A}\eta}(\widehat{A}\chi)$ for all $\chi\in\widehat{G}$.
\end{enumerate}
\end{lemma}

\proof  First we observe that $({\bf V3})$ can be conveniently rewritten as
\begin{equation}\label{V3}
\widehat{v}^{j+1}_{\eta}(\chi)=\widehat{v}^{j}_{\widehat{A}^{-1}\eta}(\widehat{A}^{-1}\chi) \qquad\text{for all } \chi,\eta \in (\ker A)^\perp.\tag{\bf V4}
\end{equation}
Define the space 
\[
V_j^{(\eta)} :=\{f\in V_j : \widehat{f}(\chi)=0 \text{ for all } \chi\not\in(\ker A^j)^{\perp}+\eta\}.
\]
Let $f\in V_j$. Then by Lemma \ref{lem:omega}
\[
f=\sum\limits_{\eta\in D(A^j)}\omega^j_{\eta}f = \sum\limits_{\eta\in D(A^j)}f_{\eta},
\]
where $f_{\eta}\in V_j^{(\eta)}$. This implies that $V_j=\bigoplus\limits_{\eta\in D(A^j)}V_j^{(\eta)}$. By MR4 of Definition~\ref{MRA} and Lemma \ref{indep} we have $\dim V_j^{(\eta)}  \ge 1$. Since $\dim V_j = m^j$ and $|D(A^j)|=m^j$, we actually have $\dim V_j^{(\eta)}  = 1$.

The proof is by the induction on scale $j$. Assume we have constructed functions $(v^j_\eta)$ for  $j=0,\ldots, j_0$, satisfying $({\bf V0})$  and $({\bf V1})$ for $j\le j_0$, and  $({\bf V2})$ and $({\bf V3})$ for $j\le j_0-1$, where $j_0 \in \N_0$. 
 Suppose first that $\widehat{v}^{j_0}_{\eta}(\chi_0)\neq 0$ for some $\chi_0\in(\ker A^{j_0+1})^{\perp}+\eta$ and $\eta \in \widehat{G}$. We set $v^{j_0+1}_{\eta}:=\omega^{j_0+1}_{\eta}v^{j_0}_{\eta}$. Then by Lemma \ref{lem:omega}
\begin{equation*}
\widehat{v}^{j_0+1}_{\eta}(\chi)=
\left\{
\begin{array}{ll}
\widehat{v}^{j_0}_{\eta}(\chi) & {\rm for}~\chi\in (\ker A^{j_0+1})^{\perp}+\eta,\\ 
0 & {\rm for}~\chi\not\in (\ker A^{j_0+1})^{\perp}+\eta.
\end{array}
\right.
\end{equation*}
Hence, $({\bf V1})$ holds for $j=j_0+1$ and $({\bf V2})$ holds for $j=j_0$. Next we check that \eqref{V3} holds for $j=j_0$.
Let $\chi, \eta\in (\ker A)^{\perp}$. If $\chi\in (\ker A^{j_0+1})^{\perp}+\eta$, then by $({\bf V2})$ and \eqref{V3} for $j=j_0-1$, we have
\[
\widehat{v}^{j_0+1}_{\eta}(\chi)=\widehat{v}^{j_0}_{\eta}(\chi)=\widehat{v}^{j_0-1}_{\widehat{A}^{-1}\eta}(\widehat{A}^{-1}\chi)=\widehat{v}^{j_0}_{\widehat{A}^{-1}\eta}(\widehat{A}^{-1}\chi).
\]
Otherwise, if $\chi\not \in (\ker A^{j_0+1})^{\perp}+\eta$, then by $({\bf V1})$ we have
\[
\widehat{v}^{j_0+1}_{\eta}(\chi)=0=\widehat{v}^{j_0}_{\widehat{A}^{-1}\eta}(\widehat{A}^{-1}\chi).
\]
Either way, \eqref{V3} holds for $j=j_0$.

Next suppose that $\widehat{v}^{j_0}_{\eta}(\chi)= 0$ for all $\chi\in(\ker A^{j_0+1})^{\perp}+\eta$ and $\eta\in(\ker A)^{\perp}$. We set $v^{j_0+1}_{\eta}(x)=v^{j_0}_{\widehat{A}^{-1}\eta}(Ax)$. 
Then,
\[
\widehat{v}^{j_0+1}_{\eta}(\chi)  = 
 \int_G v^{j_0}_{\widehat{A}^{-1}\eta}(Ax)\overline{\chi(x)}dx
 = \widehat{ D_Av^{j_0}_{\widehat{A}^{-1}\eta}}(\chi).
 \]
 By \cite[Lemma~6.6]{BR}, the right hand side is equal to $\widehat{v}^{j_0}_{\widehat{A}^{-1}\eta}(\widehat{A}^{-1}\chi)$ for $\chi\in (\ker A)^{\perp}$ and $0$ otherwise. This proves that \eqref{V3} holds for $j=j_0$.
Likewise, $({\bf V1})$ holds for $j=j_0+1$ by the inductive assumption and $({\bf V2})$ need not be verified.
 
Finally, suppose $\widehat{v}^{j_0}_{\eta}(\chi)= 0$ for $\chi\in(\ker A^{j_0+1})^{\perp}+\eta$ and $\eta\not\in (\ker A)^{\perp}$. In this case we take for $v^{j_0+1}_{\eta}$, $\eta\in D(A^{j_0+1})$, any nonzero element from the space $V^{(\eta)}_{j_0+1}$, and then let $v^{j_0+1}_{\eta'}=v^{j_0+1}_{\eta}$ if $\eta -\eta' \in (\ker A^{j_0+1})^\perp$. Since $\widehat{v}^{j_0+1}_{\eta}(\chi) = 0$ for all $\chi\not\in(\ker A^{j_0+1})^{\perp}+\eta$ we have $({\bf V1})$, while $({\bf V2})$ and \eqref{V3} need not be checked. 

Finally, observe that by the construction all functions $v_\eta^j \in V_j^{(\eta)}$ are non-zero and $v^j_\eta=v^j_{\eta'}$ if $\eta -\eta' \in (\ker A^j)^\perp$. Since $V_j=\bigoplus\limits_{\eta\in D(A^j)}V_j^{(\eta)}$ and $\dim V_j^{(\eta)}  = 1$, we conclude that $({\bf V0})$ holds as well.
\qed

\begin{proposition}
Let $(V_j)_{j\in\N_0}$ be an MRA of $L^p(G)$, $1\le p<\infty$. Let $(v^j_{\eta})_{\eta\in D(A^j)}$ be a basis of $V_j$ given by Lemma \ref{lem:basis}. A sequence $(\varphi_j)_{j=0}^{\infty}\subset L^p(G)$ is a scaling sequence if and only if
\begin{equation}\label{eq:scaling}
\varphi_j=\sum\limits_{\eta\in D(A^j)}\alpha^j_{\eta}v^j_{\eta}, 
\end{equation}
where $\alpha^j_{\eta}\neq 0$ for all $\eta\in D(A^j)$.
\end{proposition}

\proof 
Suppose $(\varphi_j)_{j=0}^{\infty}\subset L^p(G)$ is a scaling sequence. By Lemma \ref{lem:basis} we can write $\varphi_j$ as in \eqref{eq:scaling}. By Lemma \ref{lem:omega} we have
\[
\omega^j_\eta \varphi_j = \alpha^j_\eta v^j_\eta \qquad \eta \in D(A^j).
\]
By Lemma \ref{indep} we have $\alpha^j_\eta \ne 0$.

Conversely, suppose $\varphi_j$ is given by \eqref{eq:scaling}, where $\alpha^j_\eta \ne 0$. Then by Lemma \ref{indep} the functions $T_a \varphi_j$, $a \in \ker A^j$, are linearly independent, and hence a basis of $V_j$ since $\dim V_j = m^j$.
\qed

\begin{corollary}\label{cor:scaling}
If $(\varphi_j)_{j=0}^{\infty}$ is a scaling sequence, then $\omega^j_{\eta}\varphi_j=\alpha^j_{\eta}v^j_{\eta}$, where $\alpha^j_{\eta}\neq 0$. In particular, the functions $(\omega^j_{\eta}\varphi_j)_{\eta\in D(A^j)}$ form a basis of the space $V_j$.
\end{corollary}

We are now ready to give the proof of Theorem~\ref{thm:scaling}.

\begin{proof}[Proof of Theorem~\ref{thm:scaling}.] Assume that $(\varphi_j)_{j=0}^{\infty}$ is a scaling sequence for an MRA $(V_j)_{j=0}^{\infty}$ of $L^p(G)$. Part $(1)$ of Theorem~\ref{thm:scaling} follows from Lemma~\ref{lem:V0}. For $(2)$, we use Corollary~\ref{cor:scaling} noting that for $\chi\in (\ker A^j)^{\perp}+\eta$, $\eta\in\widehat{G}$,
\[
\widehat{\varphi}_j(\chi)=\widehat{\omega^j_{\eta}\varphi_j}(\chi)=\alpha^j_{\eta}\widehat{v^j_{\eta}}(\chi).
\]
By ${\bf (V1)}$ of Lemma~\ref{lem:basis} there exists $\chi\in (\ker A^j)^{\perp}+\eta$ such that $\widehat{v}^j_{\eta}(\chi)\neq 0$. Since $\alpha^j_{\eta}\neq 0$, we get $\widehat{\varphi}_j(\chi)\neq 0$. To prove $(3)$, suppose on the contrary that $\widehat{\varphi}_j(\chi)= 0$ for all $j\in\N_0$. This contradicts the axiom MR2 of Definition~\ref{MRA}. To prove $(4)$, take any $\eta\in\widehat{G}$. First consider the case $\widehat{\varphi}_{j-1}(\chi_0)\neq 0$ for some $\chi_0\in (\ker A^j)^{\perp}+\eta$. Using Lemma~\ref{lem:omega} and  Corollary~\ref{cor:scaling} we have 
\[
\omega^j_{\eta}\varphi_j=\alpha^j_{\eta}v^j_{\eta}
\qquad\text{and}\qquad  \omega^j_{\eta}\varphi_{j-1}=\omega^j_{\eta}\omega^{j-1}_{\eta}\varphi_{j-1}=\alpha^{j-1}_{\eta}\omega^j_{\eta}v^{j-1}_{\eta}
\]
for some $\alpha^j_{\eta}, \alpha^{j-1}_{\eta}\neq 0$. By ${\bf (V2)}$ of Lemma~\ref{lem:basis}, for $\chi\in (\ker A^j)^{\perp}+\eta$
\[
\frac{\widehat{\varphi}_{j-1}(\chi)}{\alpha^{j-1}_{\eta}}=\frac{\widehat{\varphi}_j(\chi)}{\alpha^j_{\eta}}.
\]
The above expression implies
\[
\widehat{\varphi}_{j-1}(\chi)=\mu^j_{\eta}\widehat{\varphi}_j(\chi),
\]
where $\mu^j_{\eta}=\frac{\alpha^{j-1}_{\eta}}{\alpha^j_{\eta}}$. In the case when $\widehat{\varphi}_{j-1}(\chi)=0$ for all $\chi\in (\ker A^j)^{\perp}+\eta$, we take $\mu^j_{\eta}=0$. 

To prove $(5)$, we use Lemma \ref{lem:omega}
\begin{equation*}
\widehat{\omega^{j+1}_{\widehat A\eta}f}(\widehat{A}\chi)=
\left\{
\begin{array}{ll}
\widehat{f}(\widehat{A}\chi), & {\rm for}~\chi\in (\ker A^j)^{\perp}+\eta,\\ 
0, & {\rm otherwise}.
\end{array}
\right.
\end{equation*}
We again use Corollary~\ref{cor:scaling}. For any $\chi\in (\ker A^j)^{\perp}+\eta$, we have $\widehat{\varphi}_j(\chi)=\alpha^j_{\eta}\widehat{v}^j_{\eta}(\chi)$ and $\widehat{\varphi}_{j+1}(\widehat{A}\chi)=\alpha^{j+1}_{\widehat{A}\eta}\widehat{v}^{j+1}_{\widehat{A}\eta}(\widehat{A}\chi)$, where $\alpha^j_{\eta}$, $\alpha^{j+1}_{\eta}\neq 0$. By ${\bf (V3)}$ of Lemma~\ref{lem:basis}, it follows that 
\[
\widehat{\varphi}_{j+1}(\widehat{A}\chi)=\frac{\alpha^{j+1}_{\widehat{A}\eta}}{\alpha^j_{\eta}}\widehat{\varphi}_j(\chi).
\]
Hence (5) holds with $\gamma^j_{\eta}=\frac{\alpha^{j+1}_{\widehat{A}\eta}}{\alpha^j_{\eta}}$.

For the sufficiency part let us assume that functions $\varphi_j\in L^p(G)$ satisfy properties (1)--(5) of Theorem \ref{thm:scaling}. Set $V_j={\rm span}\{T_a\varphi_j: a\in \ker A^j\}$. Our aim is to show that $(V_j)_{j=0}^{\infty}$ is an MRA and $(\varphi_j)$ is a scaling sequence. 

MR4 follows by Lemma \ref{lem:omega}, Lemma \ref{indep}, and property (2). 
MR3 follows then from MR4. Indeed,  write $f\in V_j$ as
\[
f=\sum\limits_{k\in \ker A^j}\alpha_kT_k\varphi_j.
\]
For $a\in \ker A^j$,
\[
T_{a}f=\sum\limits_{k\in \ker A^j}\alpha_kT_{a}T_k\varphi_j
= \sum\limits_{k\in \ker A^j}\alpha_kT_{a+k}\varphi_j.
\]
Hence, $T_{a}f\in V_j$. 

To prove MR1, we restrict ourself to a basis function $\omega^j_{\eta}\varphi_j$, $\eta\in D(A^j)$. We need to verify that if $f\in V_j$, then $f\in V_{j+1}$. By Lemma~\ref{lem:omega} we have
\begin{equation}\label{eq:dec}
\omega^j_{\eta}\varphi_j = \sum\limits_{\pi\in D(A)}\omega^{j+1}_{\eta+\widehat{A}^j\pi}\omega^j_{\eta}\varphi_j.
\end{equation}
Using property $(4)$ we can write
\begin{equation}\label{dec}
\omega^j_{\eta}\varphi_j = \sum\limits_{\pi\in D(A)}\mu^{j+1}_{\eta+\widehat{A}^j\pi}\omega^{j+1}_{\eta+\widehat{A}^j\pi}\varphi_{j+1}.
\end{equation}
Hence, by Lemma~\ref{lem:omega} we have $\omega^{j+1}_{\eta+\widehat{A}^j\pi}\varphi_{j+1}\in V_{j+1}$, which proves MR1.

Next we claim that there exists a family of functions $v^j_\eta$, $j\in \N_0$, $\eta \in \widehat{G}$ satisfying the conditions $\bf (V0)$, $\bf (V1)$, $\bf (V2)$ and $\bf (V3)$ of Lemma \ref{lem:basis}. Observe that by properties (2), (4), and (5), we have
\[
\mu^j_\eta = \mu^j_{\eta'} \quad\text{and}\quad
\gamma^j_\eta = \gamma^j_{\eta'}\qquad\text{if }\eta -\eta' \in (\ker A^j)^\perp, \ j\in \N_0.
\]
We define numbers $\alpha^j_{\eta}$, $j\in \N_0$, $\eta\in \widehat G$, recursively with respect to $j$. Set $\alpha_0^0 := 1$. Define
\[
\alpha^j_{\eta} = \begin{cases} \alpha^{j-1}_{\eta}/\mu^j_{\eta} & \mu^j_{\eta}\neq 0,\\
\alpha^{j-1}_{\widehat A^{-1}\eta}\gamma^{j-1}_{\widehat A^{-1}\eta} & \mu^j_{\eta}=0 \text{ and } \eta \in (\ker A)^\perp,
\\
1 & \mu^j_{\eta}=0 \text{ and } \eta \not\in (\ker A)^\perp.
\end{cases}
\]
By construction 
\[
\alpha^j_\eta = \alpha^j_{\eta'} \ne 0 \qquad\text{if }\eta -\eta' \in (\ker A^j)^\perp, \ j\in \N_0.
\]
Set $v^j_{\eta}=\frac{\omega^j_{\eta}\varphi_j}{\alpha^j_{\eta}}$.  Then, $(v^j_{\eta})_{\eta\in D(A^j)}$ is a basis since $(\omega^j_{\eta}\varphi_j)_{\eta\in D(A^j)}$ forms a basis of the space $V_j$ by Lemma \ref{indep} and property (2). This proves $\bf (V0)$. Likewise, we deduce that $\bf (V1)$ and $\bf (V2)$ hold. To verify 
 $\bf (V3)$ we rewrite it as \eqref{V3}.  Now, if $\mu^{j+1}_\eta =0$, then \eqref{V3} follows directly from the definition of $\alpha^{j+1}_\eta$. Otherwise, we observe the fact that 
\[
 \mu^{j+1}_\eta \ne 0 \qquad \iff \qquad \widehat v^j_{\eta}(\chi_0)\ne 0 \text{ for some }\chi_0 \in (\ker A^{j+1})^\perp + \eta.
 \]
Then, we can verify \eqref{V3} inductively in a similar way as in the proof of Lemma \ref{lem:basis}. We leave details to the reader.

To prove MR5(a), it suffices to show that it holds for the basis $(v^j_{\eta})_{\eta\in D(A^j)}$. For $\chi\in(\ker A)^{\perp}$ we have $\widehat{D_Av^j_{\eta}}(\chi)=\widehat{v}^j_{\eta}(\widehat{A}^{-1}\chi)=\widehat{v}^{j+1}_{\widehat{A}\eta}(\chi)$ by $(\bf V3)$. Otherwise, if $\chi\not\in(\ker A)^{\perp}$, then $\widehat{D_Av^j_{\eta}}(\chi)=0=\widehat{v}^{j+1}_{\widehat{A}\eta}(\chi)$ by $(\bf V1)$. This implies that $v^j_{\eta}(A\cdot)=v^{j+1}_{\widehat A \eta} \in V_{j+1}$. 

To prove MR5(b), we need to show that $Pv^{j+1}_{\eta}\in V_j$.
We claim that
\begin{equation}\label{eq:Pv}
\widehat{Pv^{j+1}_{\eta}}(\chi) = |\ker A|\widehat{v}^{j+1}_{\eta}(\widehat{A}\chi)
\end{equation}
To prove \eqref{eq:Pv} we use \eqref{homo} 
\[
\int_GPv^{j+1}_{\eta}(Ax)\chi(Ax)dx = \Delta(A)\int_GPv^{j+1}_{\eta}(x)\chi(x)dx.
\]
Using the change of variables we have
\begin{eqnarray*}
\int_GPv^{j+1}_{\eta}(Ax)\chi(Ax)dx & = & \sum\limits_{a\in \ker A}\int_Gv^{j+1}_{\eta}(x+a)\chi(Ax)dx\\
& = & \sum\limits_{a\in \ker A}\int_Gv^{j+1}_{\eta}(x)\chi(Ax)dx\\
& = & |\ker A|\widehat{v}^{j+1}_{\eta}(\widehat{A}\chi).
\end{eqnarray*}
Since $G$ is compact, $\Delta(A)=1$, which yields equation (\ref{eq:Pv}).

If $\eta \in (\ker A)^\perp$, then we use the property ({\bf V3}) of Lemma~\ref{lem:basis}, which gives $\widehat{v}^{j+1}_{\eta}(\widehat{A}\chi)=\widehat{v}^j_{\widehat{A}^{-1}\eta}(\chi)$. By  (\ref{eq:Pv}) we have $Pv^{j+1}_{\eta} = |\ker A| v^j_{\widehat{A}^{-1}\eta} \in V_j$. If $\eta \not \in (\ker A)^\perp$, then we use the property ({\bf V1}) and \eqref{eq:Pv} to get $Pv^{j+1}_{\eta} = 0$.

It only remains to prove the property MR2 of Definition \ref{MRA}. Take any $\chi \in \widehat{G}$.
By property $(3)$, there exists $j_0$ such that $\widehat{\varphi}_{j_{0}}(\chi)\neq 0$ and for $j\geq j_0$, $\widehat{\varphi}_j(\chi)\neq 0$ by property $(4)$. Hence, \eqref{omega3} yields $\widehat{v}^j_{\chi}(\chi) = \frac{\widehat{\varphi}_j(\chi)}{\alpha^j_\chi}\neq 0$ for all $j\geq j_0$. We introduce functions $h_j$ for $j\geq j_0$ by
\begin{equation}\label{kap0}
h_j(x):= 1-\frac{v^j_{\chi}(x)}{\widehat{v}^j_{\chi}(\chi)}\overline{\chi(x)}, \qquad x\in G.
\end{equation}
By taking the Fourier transform, we have
\begin{equation}\label{kap}
\widehat{h}_j(\kappa)=\int_G\overline{\kappa(x)}dx-\frac{\widehat{v}^j_{\chi}(\chi+\kappa)}{\widehat{v}^j_{\chi}(\chi)}, \qquad\kappa\in\widehat{G}.
\end{equation}
For $\kappa ={\bf 0}$, $\widehat{h}_j(\kappa)= 0$, and $\widehat{h}_j(\kappa)\neq 0$ can happen only if $\kappa\in \big(\ker A^j\big)^{\perp}$ by ({\bf V1}). 

Suppose that $f\in L^p(G)$ is such that $\widehat{f}(\kappa)= 0$ for all $\kappa \not \in (\ker A^{j_0})^\perp$. Equivalently, $f(x)=f(x+a)$ for all $a\in \ker A^{j_0}$ and $x\in G$.
For $j\ge j_0$ define
\begin{equation}\label{eqn:tj}
S_j f(x)=\frac{1}{m^{j-j_0}}\sum\limits_{[a]\in \ker A^j/\ker A^{j_0}}f(x+a),
\end{equation}
where the above sum runs over representatives of cosets of $\ker A^j/\ker A^{j_0}$ and $m=|\ker A|$. The Fourier coefficients of the function $S_jf$ can be non-zero only if $\kappa \in (\ker A^{j_0})^{\perp}$. Hence, by the fact that the dual of $\ker A^j /\ker A^{j_0}$ is $(\ker A^{j_0})^\perp/(\ker A^j)^\perp$ and \cite[Lemma (23.19)]{R} they are equal to
\begin{equation}\label{fc}
\widehat{S_jf}(\kappa)=
\frac{1}{m^{j-j_0}}\sum\limits_{[a]\in \ker A^j /\ker A^{j_0}}\kappa(a)\widehat{f}(\kappa) = \begin{cases}
\widehat f(\kappa) & \text{if } \kappa \in (\ker A^{j})^{\perp}, \\
0 & \text{otherwise.}
\end{cases}
\end{equation}
Moreover, by the triangle inequality we have 
\begin{equation}\label{tra}
||S_jf||_p \le ||f||_p.
\end{equation}

We claim that 
\begin{equation}\label{eqn:hj}
S_j h_{j_0}=h_j \qquad\text{for }j \ge j_0.
\end{equation}
Indeed, by ({\bf V2}) we have 
\[
\widehat v^{j_0}_\chi(\chi+\kappa) = \widehat v^{j}_{\chi}(\chi+\kappa) \qquad\text{for } 
\kappa \in (\ker A^{j})^{\perp}, \ j\ge j_0.
\]
Hence, by \eqref{kap} we have
\begin{equation}\label{hj}
\widehat h_{j}(\kappa) = \begin{cases}
\widehat h_{j_0}(\kappa) &\text{for } 
\kappa \in (\ker A^{j})^{\perp}, \ j\ge j_0,
\\
0 & \text{otherwise.}
\end{cases}
\end{equation}
Combining \eqref{fc} and \eqref{hj} yields \eqref{eqn:hj}.

Let $\epsilon>0$. Using the fact that trigonometric polynomials are dense in $L^p(G/(\ker A^{j_0}))$, there exists a trigonometric polynomial $q=\sum_{\kappa \in \widehat{G}} c_\kappa \kappa$ such that $c_\kappa=0$ for all $\kappa \not \in (\ker A^{j_0})^\perp$ and $||h_{j_0}-q||_p <\epsilon$. Since $\widehat h_{j_0}({\bf 0})=0$ we can additionally assume that $\widehat q({\bf 0})=0$.  By our standing assumption and Proposition~\ref{prop:ker} we have 
\[
\bigcap\limits_{j\in\N_0}(\ker A^j)^{\perp}=\{\bf 0\}.
\]
Hence, if $j$ is sufficiently large, then $S_jq$ is a zero function by \eqref{fc}.  By \eqref{tra} and \eqref{eqn:hj} we have
\[
||h_j||_p= ||S_j(h_{j_0}-q)||_p \le ||h_{j_0} - q||_p <\epsilon.
\]
 Consequently, the sequence $(h_j)$ converges to 0 as $j \to \infty$ in $L^p(G)$ norm. Thus, by \eqref{kap0} we have proved that the character function $\chi(\cdot)$ is approximated by the functions $\frac{v^j_{\chi}(\cdot)}{\widehat{v}^j_{\chi}(\chi)}\in V_j$ in $L^p(G)$ norm. This completes the proof of Theorem \ref{thm:scaling}.
\qed

\section{Construction of wavelet functions}

In this section, we are interested in constructing a wavelet orthonormal basis of $L^2(G)$. Given an MRA $(V_j)_{j\in\N_0}$ of closed subspaces of $L^2(G)$, we define the wavelet spaces 
as the orthogonal complements of $V_j$ in $V_{j+1}$ and we construct wavelet functions whose shifts form bases in these spaces. In addition, we show the existence of a special type of an MRA, called minimally supported frequency MSF MRA, for every choice of an epimorphism of a compact abelian group satisfying our standing assumptions. This yields the construction of MSF wavelets on general compact abelian groups.

\begin{proposition}\label{orth}
Let $(V_j)_{j\in\N_0}$ be an MRA of $L^2(G)$ with scaling sequence $(\varphi_j)_{j\in \N_0}$. The following are equivalent:
\begin{enumerate}[(i)]
\item the system $(T_a\varphi_j)_{a\in \ker A^j}$ is orthonormal,
\item 
the system $(m^{j/2}\omega^j_{\eta}\varphi_j)_{\eta \in D(A^j)}$ is orthonormal, where the operators $\omega^j_\eta$ are as in Definition \ref{ome} and $m=|\ker A|$,
\item we have
\begin{equation}\label{eq:ons}
\langle \omega^j_{\eta}\varphi_j, \omega^j_{\eta}\varphi_j\rangle = m^{-j} \qquad\text{for all }\eta \in D(A^j).
\end{equation}
\end{enumerate}
\end{proposition}

\proof By Lemma \ref{lem:omega} we have
\begin{equation}\label{eq:trans}
T_k\varphi_j = \sum\limits_{\eta\in D(A^j)} \omega^j_\eta(T_k\varphi_j)= \sum\limits_{\eta\in D(A^j)}\overline{\eta(k)}\omega^j_{\eta}\varphi_j.
\end{equation}
By the Plancherel formula and \eqref{omega3} for any $f,g\in L^2(G)$ we have
\begin{equation}\label{orth2}
\langle \omega^j_\eta f, \omega^j_{\eta'} g \rangle = \langle \widehat {\omega^j_\eta f}, \widehat{\omega^j_{\eta'} g}\rangle = 0 \qquad \text{for } \eta \not= \eta' \in D(A^j).
\end{equation}
Hence, for any $k, n\in \ker A^j$
\begin{eqnarray*}
\langle T_k\varphi_j, T_n\varphi_j\rangle & = & \bigg\langle \sum\limits_{\eta\in D(A^j)}\overline{\eta(k)}\omega^j_{\eta}\varphi_j, \sum\limits_{\eta'\in D(A^j)}\overline{\eta'(n)}\omega^j_{\eta'}\varphi_j\bigg\rangle\\
& = & \sum\limits_{\eta\in D(A^j)}\overline{\eta(k)}\eta(n)\langle \omega^j_{\eta}\varphi_j, \omega^j_{\eta}\varphi_j\rangle.
\end{eqnarray*}
From \cite[Lemma 23.19]{R}, we have
\begin{equation}\label{eq:eta}
	\sum\limits_{\eta\in D(A^j)}\eta(n-k)=
	\left\{
	\begin{array}{ll}
		|\ker A^j|, & {\rm for}~n=k,\\ 
		0, & {\rm otherwise}.
	\end{array}
	\right.
\end{equation}
This gives the required equation~(\ref{eq:ons}) and this argument can be reversed.
\end{proof}

From now on we will assume that $(\varphi_j)_{j\in \N_0}$ is an orthonormal scaling sequence. That is, \eqref{eq:ons} in the above proposition holds for all $j\in \N_0$.
Recall that by \eqref{dec} we have
\[
\omega^j_{\eta}\varphi_j = \sum\limits_{\pi\in D(A)}\mu^{j+1}_{\eta+\widehat{A}^j\pi}\omega^{j+1}_{\eta+\widehat{A}^j\pi}\varphi_{j+1},
\]
where coefficients $\mu^j_\eta$ are defined as in Theorem \ref{thm:scaling}. Then, by Proposition \ref{orth} and \eqref{orth2}
\begin{eqnarray*}
	m^{1-j} & = & \langle \omega^{j-1}_{\eta}\varphi_{j-1}, \omega^{j-1}_{\eta}\varphi_{j-1}\rangle\\
	& = & \bigg\langle \sum\limits_{\pi\in D(A)}\mu^{j}_{\eta+\widehat{A}^{j-1}\pi}\omega^{j}_{\eta+\widehat{A}^{j-1}\pi}\varphi_j, \sum\limits_{\pi'\in D(A)}\mu^{j}_{\eta+\widehat{A}^{j-1}\pi'}\omega^{j}_{\eta+\widehat{A}^{j-1}\pi'}\varphi_j \bigg\rangle\\
	& = & \sum\limits_{\pi\in D(A)}|\mu^{j}_{\eta+\widehat{A}^{j-1}\pi}|^2\langle \omega^{j}_{\eta+\widehat{A}^{j-1}\pi}\varphi_j, \omega^{j}_{\eta+\widehat{A}^{j-1}\pi}\varphi_j,\rangle,
\end{eqnarray*}
where the last equality is a consequence of Lemma \ref{lem:omega}.
Hence, by \eqref{eq:ons} we have
\begin{equation}\label{eq:murelation}
\sum\limits_{\pi\in D(A)}|\mu^{j}_{\eta+\widehat{A}^{j-1}\pi}|^2 = m.
\end{equation}

Now our aim is to find the wavelet spaces and wavelet bases. Let $(V_j)$ be an MRA of $L^2(G)$ and $(\varphi_j)$ be an orthonormal scaling sequence. We aim to find wavelet functions $\psi^{\nu}$, $\nu=1, \ldots, m-1$ in the space $V_{j+1}$ such that the systems $(T_{\gamma}\psi^{\nu})_{\gamma\in \ker A^j}$ are orthonormal, mutually orthogonal for different values of $\nu$, and orthogonal to the space $V_j$. To construct such functions we follow the procedure described below.

We write $D(A)=\{\pi_0, \pi_1, \ldots, \pi_{m-1} \}$, where $\pi_0=0$.
We define $b_{0k}=\mu^{j+1}_{\eta+\widehat{A}^j\pi_k}/\sqrt{m}$, where $\eta \in D(A^j)$, $k=0, \ldots, m-1$. By equation~(\ref{eq:murelation}), we have  
\[
\sum\limits_{k=0}^{m-1}|b_{0k}|^2 =1.
\]
We can extend this row to an $m\times m$ unitary matrix $B = (b_{\nu,k})_{\nu,k=0}^{m-1}$. For example, we can use Householder's transform as in \cite[(9.19)]{NPS}. We set 
\[
\alpha^{\nu,j}_{\eta+\widehat{A}^j\pi_k}=\sqrt{m}b_{\nu,k} \qquad\text{for }\nu=1, 2, \ldots, m-1, \ \eta \in D(A^j),\  k=0, \ldots, m-1.
\]
By \eqref{digit1} we have defined $\alpha^{\nu,j}_{\chi}$ for all $\chi\in D(A^{j+1})$. Then we extend this sequence to  $\widehat{G}$ by setting 
\[
\alpha^{\nu,j}_{\chi}=\alpha^{\nu,j}_{\eta}
\qquad\text{for }\chi\in (\ker A^{j+1})^{\perp}+\eta,\ \eta \in D(A^{j+1}).
\]
 Now, we define the wavelet functions $\psi^{\nu}_j$ for $\nu=1, \ldots, m-1$, in terms of Fourier transform by the formulas
\begin{equation}\label{wave1}
\widehat{\psi}^{\nu}_j(\chi)=\alpha^{\nu,j}_{\chi}\widehat{\varphi}_{j+1}(\chi) \qquad\mbox{for } \chi\in\widehat{G},
\end{equation}
and the wavelet spaces by
\[
W^{(\nu)}_j:= {\rm span}\{T_a \psi^{\nu}_j: a\in \ker A^j\}.
\]
Then we have the following theorem.

\begin{theorem}\label{wavelet}
Suppose $(V_j)_{j\in \N_0}$ is an MRA of $L^2(G)$ and $(\varphi_j)_{j\in \N_0}$ is an orthonormal scaling sequence. Then, for any $j\in \N_0$ we have
\begin{equation}\label{wavelet0}
V_{j+1} = V_j \oplus W^{(1)}_j \oplus \cdots \oplus W^{(m-1)}_j,
\end{equation}
and the system $(T_a \psi^{\nu}_j)_{a\in \ker A^j}$ is an orthonormal basis of the space $W^{(\nu)}_j$ for $\nu=1, \ldots, m-1$.
\end{theorem}

As a corollary of Theorem \ref{wavelet} and MR2 the wavelet system
\[
\{ T_a \psi^{\nu}_j: a\in \ker A^j, j\in \N_0, \nu=1,\ldots,m-1 \},
\]
together with the constant function $\varphi_0 \equiv 1$ forms an orthonormal basis of $L^2(G)$.

\proof
For any fixed $\eta \in D(A^j)$ and $j\in \N_0$, by \eqref{dec} we have
\begin{equation}\label{eq:phi}
\omega^j_{\eta}\varphi_j=\sum_{k=0}^{m-1}\mu^{j+1}_{\eta+\widehat{A}^j\pi_k}\omega^{j+1}_{\eta+\widehat{A}^j\pi_k}\varphi_{j+1}.
\end{equation}
Analogously, by Lemma \ref{lem:omega} and \eqref{wave1} we have
\begin{equation}\label{eq:psi}
\omega^j_{\eta}\psi^{\nu}_j=\sum_{k=0}^{m-1} \alpha^{\nu,j}_{\eta+\widehat{A}^j\pi_k}\omega^{j+1}_{\eta+\widehat{A}^j\pi_k}\varphi_{j+1}.
\end{equation}
In particular, \eqref{eq:psi} implies that $\psi^\nu_j \in V_{j+1}$ and hence $W^{(\nu)}_j\subset V_{j+1}$ for all $\nu=1, \ldots, m-1$. 

We claim that:
\begin{enumerate}[$(i)$]
\item $W_j^{(\nu)}\perp V_j$ for all $\nu=1, \ldots, m-1$,
\item $W_j^{(\nu)}\perp W_j^{(\kappa)}$ for all $\nu\neq \kappa$, $\nu,\kappa = 1, \ldots, m-1$. 
\end{enumerate}
For $(i)$, first note that $\sum\limits_{k=0}^{m-1}\alpha^{\nu,j}_{\eta+\widehat{A}^j\pi_k}\overline{\mu^{j+1}_{\eta+\widehat{A}^j\pi_k}}=0$ by the fact that the matrix $B$ constructed above is unitary. Using \eqref{eq:ons}, \eqref{orth2}, \eqref{eq:phi}, and \eqref{eq:psi} we have 
\begin{eqnarray*}
\langle \omega^j_{\eta}\psi^{\nu}_j, \omega^j_{\eta}\varphi_j\rangle & = & \bigg\langle \sum\limits_{k=0}^{m-1}\alpha^{\nu,j}_{\eta+\widehat{A}^j\pi_k}\omega^{j+1}_{\eta+\widehat{A}^j\pi_k}\varphi_{j+1}, \sum\limits_{k'=0}^{m-1}\mu^{j+1}_{\eta+\widehat{A}^j\pi_{k'}}\omega^{j+1}_{\eta+\widehat{A}^j\pi_{k'}}\varphi_{j+1}\bigg\rangle\\
& = & \sum\limits_{k=0}^{m-1}\alpha^{\nu,j}_{\eta+\widehat{A}^j\pi_k}\overline{\mu^{j+1}_{\eta+\widehat{A}^j\pi_k}}\langle \omega^{j+1}_{\eta+\widehat{A}^j\pi_k}\varphi_{j+1}, \omega^{j+1}_{\eta+\widehat{A}^j\pi_k}\varphi_{j+1}\rangle =0.\\
\end{eqnarray*} 
This proves $(i)$ by \eqref{orth2}.
Likewise, since $B$ is unitary, we have that $\sum\limits_{k=0}^{m-1}\alpha^{\nu,j}_{\eta+\widehat{A}^j\pi_k}\overline{\alpha^{\kappa,j}_{\eta+\widehat{A}^j\pi_k}}= m \delta_{\nu,\kappa}$ for $\nu, \kappa=1,\ldots,m-1$. Hence, 
\[
\langle \omega^j_{\eta}\psi^{\nu}_j, \omega^j_{\eta}\psi^{\kappa}_j\rangle = \sum\limits_{k=0}^{m-1}\alpha^{\nu,j}_{\eta+\widehat{A}^j\pi_k}\overline{\alpha^{\kappa,j}_{\eta+\widehat{A}^j\pi_k}}\langle \omega^{j+1}_{\eta+\widehat{A}^j\pi_k}\varphi_{j+1}, \omega^{j+1}_{\eta+\widehat{A}^j\pi_k}\varphi_{j+1}\rangle =m^{-j}\delta_{\nu,\kappa}.
\] 
This proves our claim $(ii)$. Moreover, by Proposition \ref{orth}, $(T_k\psi^{\nu}_j)_{k\in \ker A^j}$ is an orthonormal basis of $W_j^{(\nu)}$. Since
\[
\dim {V_j} = \dim {W_j^{(\nu)}} = m^j
\]
and
\[
 V_j \oplus W^{(1)}_j \oplus \cdots \oplus W^{(m-1)}_j \subset V_{j+1}
\]
the dimension count implies the equality in the above inclusion.
\qed

Next we tackle the problem of the existence of an MRA for general compact abelian groups.
Let $(\varphi_j)_{j\in \N_0}$ be a scaling sequence of an MRA. Note that by Theorem \ref{thm:scaling} the supports of $\widehat \varphi_j$ satisfy
\[
|\operatorname{supp} \widehat \varphi_j| \ge m^j \qquad\text{for all }j\in \N_0.
\]
This motivates the following definition of minimally supported frequency (MSF) multiresolution analysis.

\begin{definition}
We say that an MRA $(V_j)_{j\in \N_0}$ is MSF if its scaling sequence $(\varphi_j)_{j\in \N_0}$ satisfies
\begin{equation}\label{MSF5}
|\operatorname{supp} \widehat \varphi_j| = m^j \qquad\text{for all }j\in \N_0.
\end{equation}
\end{definition}

The following theorem characterizes all minimally supported frequency MRAs.

\begin{theorem}\label{MSF} Suppose that $(V_j)_{j\in \N_0}$ is an MSF multiresolution analysis. Then there exists a sequence $(K_j)_{j\in\N_0}$ of subsets of $\widehat G$ such that
\begin{equation}\label{MSF0}
V_j = \{ f \in L^2(G): \operatorname{supp} \widehat f \subset K_j \}
\end{equation}
satisfying for all $j\in \N_0$ the following properties:
\begin{enumerate}[(i)]
\item $K_0=\{0\}$,
\item $|K_j \cap (\eta + (\ker A^j)^\perp)|=1$ for all $\eta \in D(A^j)$,
\item $K_j \subset K_{j+1}$,
\item $\widehat A (K_j) \subset K_{j+1}$, and
\item $\bigcup_{j=0}^\infty K_j = \widehat G$.
\end{enumerate}
Conversely, if a sequence $(K_j)_{j\in\N_0}$ of subsets of $\widehat G$ satisfies $(i)$--$(v)$, then $(V_j)_{j\in \N_0}$ given by \eqref{MSF0} is an MSF MRA.
\end{theorem}

\begin{proof}
Suppose that $(V_j)_{j\in \N_0}$ is an MSF MRA. Define sets $K_j=\operatorname{supp} \widehat \varphi_j$. We claim that \eqref{MSF0} holds. Indeed, the inclusion $\subset$ in \eqref{MSF0} is trivial. By MR4 and \eqref{MSF5} the dimensions of the spaces in \eqref{MSF0} are both equal to $m^j$. Hence we have an equality in \eqref{MSF0}.

Parts (1), (2), and (3) of Theorem \ref{thm:scaling} imply $(i)$, $(ii)$, and $(v)$, respectively. By part (4) of Theorem \ref{thm:scaling}, if $\chi \in K_{j-1}$ for some $j\in \N$, then $\widehat \varphi_{j-1}(\chi) \ne 0$ implies that $\widehat \varphi_{j}(\chi) \ne 0$. This proves $(iii)$. Likewise, by part (5) of Theorem \ref{thm:scaling}, if $\chi \in K_j$, then $\widehat \varphi_{j}(\chi) \ne 0$ implies that 
$\widehat \varphi_{j+1}(\widehat A\chi) \ne 0$. Hence, $\widehat A \chi \in K_{j+1}$, which proves $(iv)$.

Conversely, if a sequence $(K_j)_{j\in\N_0}$ of subsets of $\widehat G$ satisfies $(i)$--$(v)$, then we define a sequence of functions $(\varphi_j)$ by
\begin{equation}\label{MSFv}
\widehat \varphi_j = {\bf 1}_{K_j}.
\end{equation}
Let $V_j = \operatorname{span}\{ T_a \varphi_j: a \in \ker A^j\}$. By Lemma \ref{indep} $\dim V_j=m^j$. On the other hand, by $(ii)$ we have $|K_j|=m^j$, which implies \eqref{MSF0} by the above argument. Likewise, using Theorem \ref{thm:scaling} one can verify that $(\varphi_j)_{j\in \N_0}$ is a scaling sequence.
\end{proof}

The following theorem proves that there always exist MRAs under our standing assumptions.

\begin{theorem}\label{con-MSF}
Suppose $G$ is a compact abelian group and an epimorphism $A: G \to G$ satisfies the standing assumptions: $\ker A$ is finite and $\bigcup\limits_{j\in\N_0}\ker A^j$ is dense in $G$. Then, there exists an MSF MRA $(V_j)_{j\in \N_0}$ associated with $(G,A)$.
\end{theorem}

\begin{proof}
By Theorem \ref{MSF} it suffices to construct a sequence $(K_j)_{j\in\N_0}$ of subsets of $\widehat G$ satisfying $(i)$--$(v)$. By the standing assumptions $G$ is a separable compact abelian group and hence $\widehat G$ is discrete and countable. We enumerate $\widehat G \setminus \widehat A(\widehat G)$ as $\{\chi_1,\chi_2, \ldots \}$. Define $K_0=\{0\}$. Assume that we have already defined sets $K_0,\ldots,K_{j_0}$ satisfying the following three properties for all $1\le j\le j_0$:
\begin{align}
|K_j \cap (\eta + \widehat A^j(\widehat G))| &=1 \qquad\text{ for all }\eta \in \widehat G,
\label{i1}
\\
K_{j-1} & \subset K_j,
\label{i2}
\\
\widehat A(\widehat G) \cap K_j & = \widehat A(K_{j-1}).
\label{i4}
\end{align}
Our goal is to construct a set $K_{j_0+1}$ such that \eqref{i1}--\eqref{i4} hold for $j=j_0+1$. 
Let
$K'_{j_0+1}= K_{j_0} \cup \widehat A(K_{j_0})$. We claim that
\begin{equation}\label{i5}
|K'_{j_0+1} \cap (\eta + \widehat A^{j_0+1}(\widehat G))| \le 1
\qquad\text{for all } \eta \in \widehat G.
\end{equation}
By \eqref{i1} we have
\begin{align}\notag
|K_{j_0} \cap (\eta + \widehat A^{j_0+1}(\widehat G))| & \le 1 \qquad\text{ for all }\eta \in \widehat G,
\\
\label{i6}
|\widehat A (K_{j_0}) \cap (\eta + \widehat A^{j_0+1}(\widehat G))| &
= 
\begin{cases} 1 & \eta \in \widehat A(\widehat G),\\
0 & \text{otherwise.}
\end{cases}
\end{align}
Hence, \eqref{i5} might fail only if there exist $\eta\in \widehat A(\widehat G)$, $\xi_1 \in K_{j_0}$, $\xi_2 \in \widehat A (K_{j_0})$ such that $\xi_1,\xi_2 \in \eta+\widehat A^{j_0+1}(\widehat G)$. This implies that $\xi_1 \in \widehat A(\widehat G)$. By \eqref{i2} and \eqref{i4}, 
\[
(\widehat A)^{-1}(\xi_1) \in K_{j_0-1} \subset K_{j_0}.
\]
 On the other hand, $(\widehat A)^{-1}(\xi_2) \in K_{j_0}$ and both $(\widehat A)^{-1}(\xi_1)$ and $(\widehat A)^{-1}(\xi_2)$ belong to the same coset of $\widehat G/\widehat A^{j_0}(\widehat G)$. Hence, by \eqref{i1}, we have $\xi_1=\xi_2$, which proves \eqref{i5}.
 
By \eqref{i5} and \eqref{i6} we have
\begin{equation}\label{i7}
|K'_{j_0+1} \cap (\eta + \widehat A^{j_0+1}(\widehat G)| = 1
\qquad\text{for all } \eta \in \widehat A(\widehat G).
\end{equation}
Now we find the smallest $m\in \N$ such that 
\[
K'_{j_0+1} \cap (\chi_m + \widehat A^{j_0+1}(\widehat G)) =\emptyset.
\]
Then, we find the smallest $m'\in \N$ such that
\[
(K'_{j_0+1}\cup\{\chi_m\}) \cap (\chi_{m'} + \widehat A^{j_0+1}(\widehat G)) =\emptyset,
\]
and we keep adding minimal elements from $\widehat G \setminus \widehat A(\widehat G)$ until we have constructed the set $K_{j_0+1}=K'_{j_0+1} \cup \{\chi_m,\chi_{m'},\ldots\}$ such that
\eqref{i1} holds for $j=j_0+1$. This will happen after a finite number of steps. The property \eqref{i2} holds for $j=j_0+1$ by the definition of $K_{j_0+1}$. By the construction of the set $K_{j_0+1}$ and the inductive hypotheses \eqref{i2} and \eqref{i4}, we have
\[
\begin{aligned}
K_{j_0+1} \cap \widehat A(\widehat G) = (K_{j_0} \cup \widehat A(K_{j_0})) \cap \widehat A(\widehat G) & = 
\widehat A(K_{j_0}) \cup (\widehat A(\widehat G) \cap K_{j_0})
\\
&= \widehat A(K_{j_0}) \cup \widehat A(K_{j_0-1})
= \widehat A(K_{j_0}).
\end{aligned}
\]
This proves \eqref{i4} for $j=j_0+1$, and completes the inductive step. Therefore, we have constructed sets $(K_j)$ satisfying $(i)$--$(iv)$ in Theorem \ref{MSF}.

Finally, the property $(v)$ follows by the choice of minimal elements in the above construction. Indeed, suppose that there exists an element $\chi_m$ which was never chosen. That is, there exists $j_0 \in \N$ such that $\chi_1,\ldots,\chi_{m-1} \subset K_{j_0}$ and $\chi_m \not\in K_j$ for all $j\ge j_0$. Let $\chi \in K_{j_0}$ be such that $\chi_m \in \chi+ \widehat A^{j_0}(\widehat G)$. Then, by our construction we have 
\begin{equation}\label{i9}
\chi_m \in \chi+ \widehat A^{j}(\widehat G) \qquad\text{for all  } j \ge j_0.
\end{equation}
 This is shown inductively using the fact that $\chi \not\in \widehat A(\widehat G)$ and the decomposition 
\begin{equation}\label{i8}
\chi+ \widehat A^{j}(\widehat G) = \bigcup_{\pi \in D(A)} \chi+ \widehat A^j \pi + \widehat A^{j+1}(\widehat G).
\end{equation}
Indeed, suppose that $\chi_m \in \chi+ \widehat A^{j}(\widehat G)$. Let $\pi \in D(A)$ be such that
\[
\chi_m \in \chi+ \widehat A^j \pi + \widehat A^{j+1}(\widehat G).
\]
Then, by the construction of $K_{j+1}$ we have
\[
\emptyset \ne K_{j+1}' \cap (\chi+ \widehat A^j \pi + \widehat A^{j+1}(\widehat G)) = K_j \cap (\chi+ \widehat A^j \pi + \widehat A^{j+1}(\widehat G)).
\]
Since $\chi \in K_{j_0} \subset K_j$, by \eqref{i1} and \eqref{i8}, the above intersection is a singleton $\{\chi\}$. Hence, we have $\pi=0$ and thus $\chi_m \in \chi+ \widehat A^{j+1}(\widehat G)$, which proves \eqref{i9}.
By the fact that 
\[\bigcap_{j=0}^\infty \widehat A^j(\widehat G) = \{\bf 0\}
\]
and \eqref{i9} we have $\chi=\chi_m$, which is a contradiction. This proves $(v)$ and completes the proof of Theorem \ref{con-MSF}.
\end{proof}

We finish the paper by illustrating how Theorem \ref{wavelet} can be applied in the context of an MSF MRA given by Theorem \ref{con-MSF} to produce orthonormal MSF wavelets.

\begin{theorem}\label{msfw} Suppose that $(V_j)_{j\in \N_0}$ is an MSF MRA associated with $(G,A)$ as in Theorem \ref{con-MSF}. Let $m=|\ker A|$. Then there exists wavelet functions $\psi^\nu_j$, $j\in \N_0$, $\nu=1,\ldots,m-1$, such that  $(T_a \psi^{\nu}_j)_{a\in \ker A^j}$ is an orthonormal basis of spaces $W^{(\nu)}_j$ satisfying \eqref{wavelet0} and each $\psi^\nu_j$ has minimal support in frequency
\begin{equation}\label{msfw0}
|\operatorname{supp} \widehat \psi^\nu_j| = m^j \qquad\text{for all }j\in \N_0,\ \nu=1,\ldots,m-1.
\end{equation}
\end{theorem}

\begin{proof}
Recall that the spaces $V_j$ are of the form \eqref{MSF0} for some sequence $(K_j)_{j\in\N_0}$ of subsets of $\widehat G$ satisfying conditions $(i)$--$(v)$ of Theorem \ref{MSF}. Moreover, by \eqref{MSFv} we can assume that the sequence of scaling functions $(\varphi_j)_{j\in \N_0}$ is orthonormal and given by 
\begin{equation}\label{MSFo}
\widehat \varphi_j = m^{-j/2}{\bf 1}_{K_j}.
\end{equation}
We can then follow the general construction procedure of Theorem \ref{wavelet} by observing that the first row of the $m\times m$ matrix $B$ contains exactly one non-zero entry, which is equal to $1$. To guarantee that this matrix is unitary it suffices to choose for $B$ a permutation matrix. Then, one can show that the wavelets defined by \eqref{wave1} satisfy \eqref{msfw0}.

Alternatively, we can give a more direct construction of wavelet functions as follows. Since
\[
(\ker A^j)^\perp =\widehat A^{j}(\widehat G) = \bigcup_{\pi \in D(A)}
\widehat A^j \pi + \widehat A^{j+1}(\widehat G),
\]
by $(ii)$ of Theorem~\ref{MSF} we have for all $j\in \N_0$,
\[
|K_{j+1} \cap (\eta + \widehat A^{j}(\widehat G)| = m \qquad\text{for all }\eta \in D(A^j).
\]
Hence, we can find disjoint sets $K^{(0)}_j, \ldots, K^{(m-1)}_j$ such that $K^{(0)}_j=K_j$ and
\begin{equation}\label{MS4}
K_{j+1} = K_j \cup K^{(1)}_j\cup \ldots \cup K^{(m-1)}_j
\end{equation}
and 
\begin{equation}\label{MS5}
|K^{(\nu)}_j \cap (\eta + \widehat A^{j}(\widehat G))| = 1 \qquad\text{for all }\eta \in D(A^j), \ \nu=1,\ldots,m-1.
\end{equation}
Define wavelet functions $\psi^\nu_j$ by
\begin{equation}\label{MS6}
\widehat \psi^\nu_j = m^{-j/2}{\bf 1}_{K^{(\nu)}_j}.
\end{equation}
By Proposition \ref{orth} $(T_{a}\psi^{\nu}_j)_{a\in \ker A^j}$ is an orthonormal basis of spaces 
\[
W^{(\nu)}_j = \{ f \in L^2(G): \operatorname{supp} \widehat f \subset K^{(\nu)}_j \}
\]
satisfying \eqref{wavelet0} by \eqref{MS4}. Finally, \eqref{msfw0} follows immediately from \eqref{MS5} and \eqref{MS6}.
\end{proof}

\end{document}